\documentclass[10pt]{article}

\usepackage{psfrag}
\usepackage{graphicx}
\usepackage{amsmath,amsfonts,amssymb,amsxtra,amsthm}
\usepackage{mathrsfs}
\usepackage{stmaryrd}
\usepackage{pinlabel}
\usepackage{xypic}
\usepackage{color}

\usepackage{hyperref}

\input xy
\xyoption{all}

\def\immerses{\looparrowright}
 \def\into{\hookrightarrow}

\def\RR{\mathbb{R}}
\def\BB{\mathcal{B}}
\def\AA{\mathcal{A}}

\def\zee{\mathbb{Z}}

\def\ess{\ensuremath{\mathbb{S}}}

\newcommand{\mnote}[1]{}

\newtheorem{theorem}{Theorem}%[section]

\newtheorem{lemma}[theorem]{Lemma}

\newtheorem{corollary}[theorem]{Corollary}

\newtheorem{conjecture}[theorem]{Conjecture}

\theoremstyle{definition}
\newtheorem{definition}[theorem]{Definition}

\theoremstyle{remark}
\newtheorem{remark}[theorem]{Remark}

\author{Larsen Louder and Henry Wilton} 

\title{Stackings and the $W$-cycles conjecture}

\begin{document}
\maketitle

\begin{abstract}
  We prove Wise's $W$-cycles conjecture.  Consider a compact graph
  $\Gamma'$ immering into another graph $\Gamma$.  For any immersed
  cycle $\Lambda:S^1\to \Gamma$, we consider the map $\Lambda'$ from
  the circular components $\ess$ of the pullback to $\Gamma'$. Unless
  $\Lambda'$ is reducible, the degree of the covering map $\ess\to
  S^1$ is bounded above by minus the Euler characteristic of
  $\Gamma'$.  As a consequence, we obtain a homological version of
  coherence for one-relator groups.
\end{abstract}

\section{Introduction}

As part of his work on the coherence of one-relator groups, Wise made
a conjecture about the number of lifts of a cycle in a free group
along an immersion, which we will call the \emph{$W$-cycles
  conjecture}.  If $f_1:\Gamma_1\immerses\Gamma$ and
$f_2:\Gamma_2\immerses\Gamma$ are immersions of graphs, then the fibre
product
\[
\Gamma_1\times_\Gamma\Gamma_2=\{(x,y)\in \Gamma_1\times \Gamma_2\mid f_1(x)=f_2(y)\}
\]
immerses into $\Gamma_1$ and $\Gamma_2$, and is the pullback of $f_1$
and $f_2$.  An immersed loop $\Lambda:S^1\immerses\Gamma$ is
\emph{primitive} if it does not factor properly through any other
immersion $S^1\immerses\Gamma$.

With this definition, the $W$-cycles conjecture can be stated as follows.

\begin{conjecture}[Wise \cite{wise}]\label{conjecture}
  Let $\rho:\Gamma'\to\Gamma$ be an immersion of finite connected core
  graphs and let $\Lambda:S^1\to \Gamma$ be a primitive immersed loop.
  Let $\ess$ be the union of the circular components of
  $\Gamma'\times_\Gamma S^1$.  Then the number of components of $\ess$
  is at most the rank of $\Gamma'$.
\end{conjecture}

The purpose of this note is to prove Wise's conjecture; indeed, we
prove a stronger statement.  As usual, if $\pi$ is a covering map then
$\deg\pi$ denotes its degree, the number of preimages of a point.  An
immersion of a union of circles $\Lambda:\ess\to\Gamma$ is called
\emph{reducible} if there is an edge of $\Gamma$ which is traversed at
most once by $\Lambda$.

\begin{theorem}
  \label{maintheorem}
  Let $\rho:\Gamma'\immerses\Gamma$ be an immersion of finite
  connected core graphs and let $\Lambda:S^1\to \Gamma$ be a primitive
  immersed loop.  Suppose that $\ess$, the union of the circular
  components of $\Gamma'\times_\Gamma S^1$, is non-empty, so there is
  a natural covering map $\sigma:\ess\immerses S^1$.  Then either
  \[
  \deg\sigma \leq -\chi(\Gamma')
  \]  
  or the pullback immersion $\Lambda':\ess\to\Gamma'$ is reducible.
\end{theorem}

The statement of the conjecture is a corollary of this
theorem. Indeed, the inequality in the theorem is strictly stronger
than the inequality in the conjecture; alternatively, in the reducible
case, we may remove an edge and proceed by induction.

The connection with Baumslag's question is provided by Wise's notion
of nonpositive immersions. As in the case of graphs, an immersion of
cell complexes is a locally injective cellular map.

\begin{definition}[Wise]
  A cell complex $X$ has \emph{nonpositive immersions} if, for every
  immersion of compact, connected complexes $Y\immerses X$, either
  $\chi(Y)\leq 0$ or $Y$ has trivial fundamental group.
\end{definition}

Our main theorem implies that the presentation complexes associated to
one-relator groups have non-positive immersions.

\begin{corollary}
  Let $X$ be compact 2-complex with one 2-cell $e^2$ and suppose that
  the attaching map of $e^2$ is an immersion.  Then $X$ has
  non-positive immersions.
\end{corollary}

\begin{proof}
  Suppose that $Y$ is a compact, connected 2-complex and
  $\rho:Y\immerses X$ is an immersion.  If $Y$ has no 2-cells then the
  result is immediate.  If some 1-cell of $Y$ is traversed exactly
  once by the 2-cells of $Y$ then the result follows by induction on
  the number of 2-cells of $Y$.  Therefore, we may assume that each
  1-cell of $Y$ is traversed at least twice by the 2-cells of $Y$.
  Set $\Gamma=X^{(1)}$ and $\Gamma'=Y^{(1)}$.  Let $\Lambda:S^1\to
  \Gamma$ be the attaching map of $e^2$ and let
  $\Lambda':\ess\to\Gamma'$ be the pullback map.  Since each 1-cell of
  $\Gamma'$ is traversed at least twice by $\Gamma'$, it follows from
  Theorem \ref{maintheorem} that the number of 2-cells of $Y$ is
  bounded above by $-\chi(\Gamma')$, and hence $\chi(Y)\leq 0$ as
  required.
\end{proof}

Wise has conjectured that, if a 2-complex $X$ has nonpositive
immersions, then its fundamental group is coherent.  Although
Baumslag's conjecture remains open, we do obtain a weaker statement:
every finitely generated subgroup of a one-relator group has finitely
generated second homology.

\begin{corollary}
  Let $G$ be a torsion-free%\footnote{What about the case with
%    torsion?} 
one-relator group. If $H<G$ is finitely generated
  then \[b_2(H)\leq b_1(H)-1\]
\end{corollary}

\begin{proof}
  Let $X$ be the presentation complex for $G$ and let $Y\immerses X$
  be a covering map corresponding to $H$.  Choose an exhaustion of $Y$
  by finite complexes
  \[
  Y_1\subseteq Y_2\subseteq \ldots \subseteq Y_n\subseteq \ldots \subseteq Y
  \]
  such that each inclusion $Y_i\into Y_{i+1}$ induces a surjection on
  fundamental groups. No $Y_i$ is simply connected and so, since $X$
  has non-positive immersions, $\chi(Y_i)\leq 0$ for all $i$.  Since
  homology commutes with limits, it follows that $b_2(Y)\leq b_1(Y)-1$
  as claimed.  Since $Y$ is aspherical by Lyndon's theorem
  \cite{lyndon_cohomology_1950}, the claimed inequality for $H$
  follows.
\end{proof}

Our proof of Theorem~\ref{maintheorem} was inspired by the proof of
the following theorem of Duncan and Howie. In particular, the punch
line in Lemma~\ref{lem:reducedrankofsubgroup} is essentially their
proof of \cite[Lemma~3.1]{duncanhowie}.

The \emph{genus} of an element $w$ in a free group $F$ is the minimal
number $g$ so that $w=\Pi_{i=1}^g\left[x_i,y_i\right]$ has a solution
in $F$, or equivalently, the minimal genus of a once-holed surface
mapping into a graph representing $F$ with boundary $w$.

\begin{theorem}[{\emph{c.f.}~\cite[Corollary~5.2]{duncanhowie}}]
  \label{thm:duncanhowie}
  Let $w$ be an indivisible element in a free group $F$. Then the
  genus of $w^m$ is at least $m/2$.
\end{theorem}

Theorem~\ref{thm:duncanhowie} of course follows from
Theorem~\ref{maintheorem}. Indeed, let $\Gamma$ be a graph
representing $F$ and consider a map from a one-holed surface $\Sigma$
of genus $g$ to $\Gamma$ with boundary $w^m$.  After modifying the map
by a homotopy, we may assume that the preimages of midpoints of edges
are properly embedded arcs and circles in $\Sigma$; in particular,
$w^m$ traverses each edge either at least twice or not at all.
Consider an immersion $\Gamma'\immerses\Gamma$ representing the image
of $\pi_1(\Sigma)$ in $G=\pi_1(\Gamma)$, so the map $\Sigma\to\Gamma$
lifts to $\Gamma'$.  By removing edges inductively, we may assume that
$\Gamma'$ is irreducible.  But now
\[
m\leq -\chi(\Gamma')\leq-\chi(\Sigma)
\]
where the first inequality follows from Theorem \ref{maintheorem} and
the second follows the fact that the map from $\Sigma$ to $\Gamma'$ is
surjective on fundamental groups.

\iffalse

We also obtain a completely transparent proof of Magnus'
Freiheitssatz. See Section~\ref{sec:freiheitssatz}. \marginpar{how do
  we sex this up?}

\begin{theorem}[Freiheitssatz]
  \label{thm:freiheitssatz}
  Let $\Lambda\colon S^1\to\Gamma$ be an immersion representing an
  element $g\in F=\pi_1(\Gamma)$. If $\Omega\colon S^1\to\Gamma$ is an
  immersed loop representing an element in the normal closure of $g$
  then $\Omega$ traverses every edge that $\Lambda$ traverses.
\end{theorem}

Something about Magus, staggered, Lyndon and his ``maximum
principle'', Howie and the tower. All this stuff is really just slick
repackaging.

\fi

While this work was in preparation, we learned that Helfer and Wise have also proved the $W$-cycles conjecture \cite{helfer_counting_2014}.

\section{Stackings}

\subsection{Computing the characteristic of a free group}

\begin{definition}
  Let $\Gamma$ be a finite graph, let $\mathbb{S}$ be a disjoint union
  of circles, and let $\Lambda\colon\mathbb{S}\immerses\Gamma$ be a
  continuous map.  Consider the trivial $\RR$-bundle
  $\pi\colon\Gamma\times \RR\to \Gamma$. A \emph{stacking} is an
  embedding $\hat\Lambda\colon\ess\into\Gamma\times\RR$ such that
  $\pi\hat\Lambda=\Lambda$.
\end{definition}

Although this definition is very simple, it leads to a natural way of
estimating the Euler characteristic of a graph.

Let $\pi$ and $\iota$ be the projections of $\Gamma\times\RR$ to
$\Gamma$ and $\RR$, respectively. Let
\[
\AA_{\hat\Lambda}=\{x\in \ess\mid\forall y\neq x\,\,
(\Lambda(x)=\Lambda(y)\Rightarrow\iota(\hat\Lambda(x))>\iota(\hat\Lambda(y)))\}
\]
and
\[
\BB_{\hat\Lambda}=\{x\in\ess\mid\forall y\neq x\,\,
(\Lambda(x)=\Lambda(y)\Rightarrow\iota(\hat\Lambda(x))<\iota(\hat\Lambda(y)))\}
\]
Intuitively, $\AA_{\hat\Lambda}$ is the set of points of
$\hat{\Lambda}(\ess)$ that one sees if one looks at
$\hat{\Lambda}(\ess)$ from above, and likewise $\BB_{\hat\Lambda}$
is the set of points of $\hat{\Lambda}(\ess)$ that one sees from
below.

Henceforth, assume that $\Lambda:\ess\to\Gamma$ is an immersion.  The
stacking $\hat\Lambda$ is called \emph{good} if $\AA_{\hat\Lambda}$
and $\BB_{\hat\Lambda}$ each meet every connected component of
$\mathbb{S}$.  For brevity, we will call a subset $s\subseteq \ess$ an
\emph{open arc} if it is connected, simply connected, open, and a
union of vertices and interiors of edges.

\begin{lemma}\label{areopenarcs}
  If $\Lambda$ is an immersion then each connected component of
  $\AA_{\hat\Lambda}$ or $\BB_{\hat\Lambda}$ is either a connected
  component of $\mathbb{S}$ or an open arc in $\mathbb{S}$.
\end{lemma}

\begin{proof}
  It suffices to prove the lemma for $\AA_{\hat\Lambda}$.  Let
  $s\subseteq\ess$ be a connected component of $\AA_{\hat\Lambda}$.
  It follows from the definition that $s$ is open. Note also that if
  one point $p$ in the interior of an edge $e$ is contained in
  $\AA_{\hat\Lambda}$ then the whole interior of $e$ is contained in
  $\AA_{\hat\Lambda}$. This completes the proof.
\end{proof}

The next lemma characterizes reducible maps in terms of a stacking; in
particular, reducibility is reduced to non-disjointness of
$\AA_{\hat{\Lambda}}$ and $\BB_{\hat{\Lambda}}$.

\begin{lemma}\label{reducible}
  If $\hat{\Lambda}$ is a stacking of an immersion
  $\Lambda:\ess\to\Gamma$, then $\AA_{\hat{\Lambda}}\cap
  \BB_{\hat{\Lambda}}$ contains the interior of an edge if and only if
  $\Lambda$ is reducible.  If $\hat{\Lambda}$ is a good stacking and
  $\AA_{\hat{\Lambda}}$ or $\BB_{\hat{\Lambda}}$ contains a circle
  then $\hat{\Lambda}$ is reducible.
\end{lemma}

\begin{proof}
  To first assertion is immediate from the definitions.  It suffices
  to prove the second assertion for $\AA_{\hat{\Lambda}}$.  Let $S$ be
  a component of $\ess$ contained in $\AA_{\hat{\Lambda}}$. Since
  $\ess$ is good, there is an edge $e$ of $S$ contained in
  $\BB_{\hat{\Lambda}}$. Therefore, $e$ is contained in both
  $\AA_{\hat{\Lambda}}$ and $\BB_{\hat{\Lambda}}$.  It follows that
  $e$ is traversed exactly once $\hat{\Lambda}$, so $\hat{\Lambda}$ is
  reducible.
\end{proof}

The final lemma of this section is completely elementary, but is the
key observation in the proof.  It asserts that number of open arcs in
$\AA_{\Lambda}$ or $\BB_{\Lambda}$ computes the Euler characteristic
of the image of $\Lambda$.  This is illustrated in Figure
\ref{fig:computerank}.

\begin{figure}[ht]
  \centerline{\includegraphics[width=.8\textwidth]{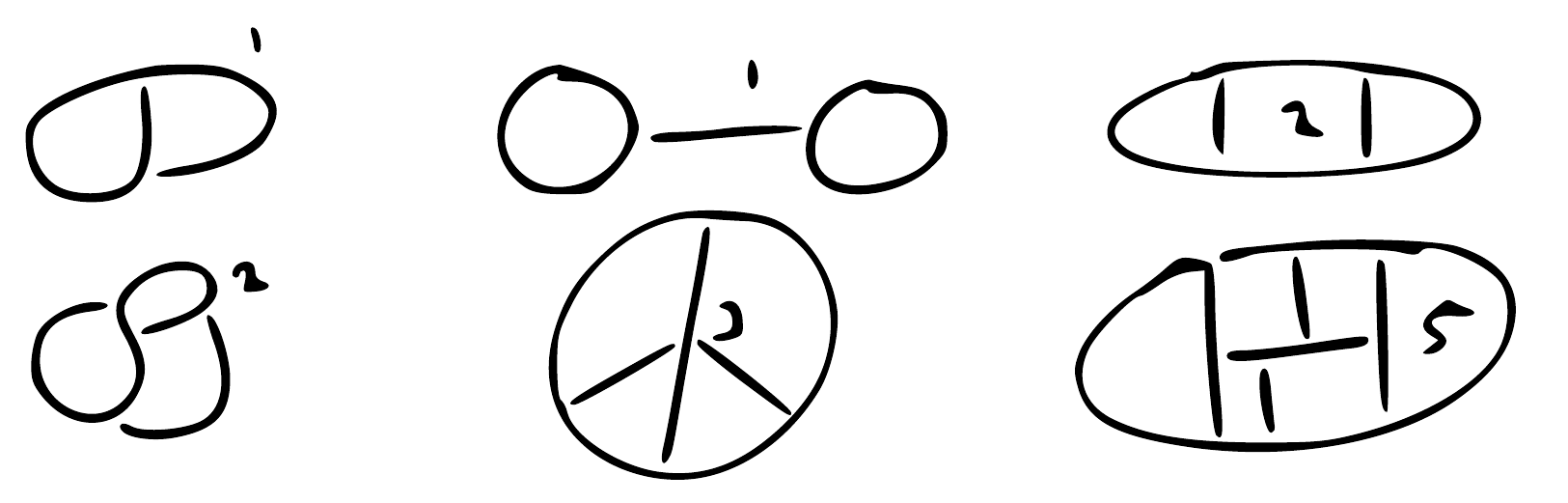}}
    \caption{How to compute the rank of a free group.}
  \label{fig:computerank}
\end{figure}

\begin{lemma}
  \label{lem::computerank}
  Let $\hat{\Lambda}:\ess\to\Gamma\times\RR$ be a stacking of a
  surjective immersion $\Lambda:\ess\to\Gamma$.  The number of open
  arcs in $\AA_{\hat\Lambda}$ or $\BB_{\hat\Lambda}$ is equal to
  $-\chi(\Gamma)$.
\end{lemma}

\begin{proof}
  As usual, it suffices to prove the lemma for $\AA_{\hat\Lambda}$.
  Let $x$ be a vertex of $\Gamma$ of valence $v(x)$.  Because
  $\Lambda$ is surjective, exactly $v-2$ edges incident at $x$ are
  covered by open arcs of $\AA_{\hat{\Lambda}}$ that end at $x$.
  Therefore, the number of open arcs is
  \[
  \frac{1}{2}\sum_{x\in V(\Gamma)} (v(x)-2)
  \]
  which is easily seen to be $-\chi(\Gamma)$.
\end{proof}

\subsection{Computing the characteristic of a subgroup}

As in the previous section, $\Gamma$ is a finite graph,
$\Lambda:\ess\immerses\Gamma$ is an immersion and
$\hat{\Lambda}:\ess\into\Gamma\times\RR$ is a stacking. Consider now
an immersion of finite graphs $\rho:\Gamma'\to\Gamma$, and let $\ess'$
be the circular components of the fibre product
$\ess\times_{\Gamma}\Gamma'$, which is equipped with a covering map
$\sigma:\ess'\to\ess$ and an immersion $\Lambda':\ess'\to\Gamma'$.  In
order to prove Theorem \ref{maintheorem}, we would like to estimate
the characteristic of $\Gamma'$ in terms of $\hat\Lambda$.

The stacking $\hat{\Lambda}$ of $\Lambda$ naturally pulls back to a
stacking $\hat{\Lambda}'$ of $\Lambda'$.  More precisely, there is a
natural isomorphism
\[
(\Gamma\times\RR)\times_\Gamma\Gamma'\cong\Gamma'\times\RR
\] 
and the universal property of the fibre bundle defines a map
$\hat{\Lambda}':\ess'\to\Gamma\times\RR$, so we have the following
commutative diagram.

\[
\xymatrix{
   & \Gamma'\times\RR  \ar@{>}[dd]^(.35){\pi'}\ar@{>}[rr]^{\hat{\rho}} & & \Gamma\times \RR \ar@{>}[dd]^\pi \\
  \ess'\ar@{>}[ur]^{\hat{\Lambda}'}\ar@{>}[dr]^{\Lambda'}\ar@{>}[rr]^(.35)\sigma  & &\ess \ar@{>}[ur]^{\hat{\Lambda}}\ar@{>}[dr]^\Lambda  &  \\
     & \Gamma'\ar@{>}[rr]^\rho	& & \Gamma 
}
\]
\begin{lemma}
  If $\hat{\Lambda}$ is a stacking then $\hat{\Lambda}'$ is also a
  stacking.  Furthermore, if $\hat{\Lambda}$ is good then
  $\hat{\Lambda}'$ is also good.
\end{lemma}

\begin{proof}
  The proof of the first assertion is a diagram chase, which we leave
  as an exercise to the reader.  The second assertion follows
  immediately from the observation that
  $\sigma^{-1}(\AA_{\hat\Lambda})\subseteq\AA_{\hat\Lambda'}$ and
  $\sigma^{-1}(\BB_{\hat\Lambda})\subseteq\BB_{\hat\Lambda'}$.
\end{proof}

The final lemma in this section estimates the Euler characteristic of
$\Gamma'$ using a stacking of the pullback immersion $\Lambda'$.
Since all finitely generated subgroups of free groups can be realized
by immersions of finite graphs, this can be thought of as an estimate
for the rank of a subgroup of a free group; this point of view
motivates the title of this subsection.

\begin{lemma}
  \label{lem:reducedrankofsubgroup}
  If $\hat\Lambda$ is a good stacking then either
  $\Lambda':\ess'\to\Gamma'$ is reducible or
  \[
  -\chi(\Lambda'(\ess'))\geq \deg\sigma
  \]
\end{lemma}

\begin{proof}
  Suppose $\Lambda'$ is not reducible. We may assume that $\Lambda'$
  is surjective.

  Let $e$ be an edge in
  $\AA_{\hat{\Lambda}}$ and consider its $\deg\sigma$ preimages
  $\{e'_j\}$.  Since $\Lambda'$ is not reducible, no component of
  $\AA_{\hat{\Lambda}'}$ is a circle, by Lemma \ref{reducible}, and so
  every $e'_j$ is contained in an open arc of $\AA_{\hat{\Lambda}'}$.

  If $-\chi(\Gamma')< \deg\sigma$ then, by Lemma
  \ref{lem::computerank} and the pigeonhole principle, two distinct
  preimages $e'_i$ and $e'_j$ are contained in the same open arc $A$.
  But then, for any $f$ an edge of $\ess$ contained in
  $\BB_{\hat{\Lambda}}$ (which again exists because $\hat\Lambda$ is
  good), $A$ also contains an edge $f'$ that maps to $f$.  Therefore,
  $\AA_{\hat{\Lambda}'}\cap \BB_{\hat{\Lambda}'}$ contains $f'$, and
  so $\Lambda'$ is reducible by Lemma \ref{reducible}.  See
  Figure~\ref{fig:ranktoosmall}.
\end{proof}

\iffalse
\begin{figure}[ht]
  \centering
  \includegraphics[width=.6\textwidth]{ranktoosmall}
  \caption{If $-\chi(\Gamma')$ is too small in comparison to the sum
    of the degrees then $\Lambda'$ is reducible.}
  \label{fig:ranktoosmall}
\end{figure}
\fi

\begin{figure}[h]
\begin{center}
 \centering \def\svgwidth{300pt}
 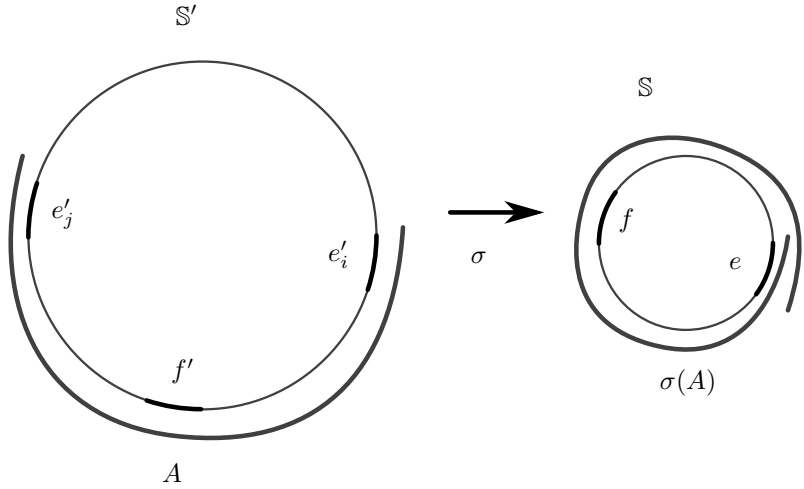
 \caption{If $-\chi(\Gamma')$ is too small in comparison to the sum
    of the degrees then $\Lambda'$ is reducible.}
  \label{fig:ranktoosmall}
  \end{center}
\end{figure}

\section{A tower argument}

In order to apply Lemma \ref{lem:reducedrankofsubgroup} to prove
Theorem \ref{maintheorem}, we need to prove that stackings exist.  The
proof here employs a \emph{cyclic tower argument} of the kind used by
Brodski{\u\i} and Howie to prove that one-relator groups are
right-orderable and locally indicable \cite{brod:order,howie:order}.

\begin{definition}
  Let $X$ be a complex.  A \emph{(cyclic) tower} is the composition of
  a finite sequence of maps
  \[
  X_0\immerses X_{1}\immerses\ldots \immerses X_n=X
  \]
  such that each map $X_i\immerses X_{i+1}$ is either an inclusion of
  a subcomplex or a covering map (resp.\ a normal covering map with
  infinite cyclic deck group).
\end{definition}

One can argue by induction with towers because of the following lemma
of Howie (building on ideas of Papakyriakopoulos and Stallings)
\cite{howie_pairs_1981}.

\begin{lemma}
  Let $Y\to X$ be cellular map of compact complexes. Then there exists
  a maximal (cyclic) tower map $X'\immerses X$ such that $Y\to X$
  lifts to a map $Y\to X'$.
\end{lemma}

As in the previous sections let $\Gamma$ be a graph.  To apply a
cyclic tower argument, one needs to know that the phenomene of
interest are preserved by cyclic coverings.  In our case, that control
is provided by the following lemma.

\begin{lemma}
  \label{lem: Cyclic embedding}
  Consider an infinite cyclic cover of a graph $\Gamma$. Then there is
  an embedding $\tilde{\Gamma}\times\RR\into\Gamma\times\RR$ such that
  the diagram
  \[
  \xymatrix{
    \tilde{\Gamma}\times \RR  \ar@{>}[r]^{\tilde{\pi}}\ar@{>}[d]  & \tilde{\Gamma}\ar@{>}[d] \\
    \Gamma\times\RR \ar@{>}[r]^\pi  & \Gamma
  }
  \]
  commutes where, as usual $\pi$ and $\tilde{\pi}$ denote coordinate
  projections onto $\Gamma$ and $\tilde{\Gamma}$ respectively.  (Note
  that the embedding $\tilde{\Gamma}\times\RR\into\Gamma\times\RR$ is
  usually \emph{not} natural with respect to the coordinate
  projections onto $\RR$.)
\end{lemma}

\begin{proof}
  Elements $g$ of the group $\pi_1\Gamma$ act by deck transformations
  $x\mapsto gx$ on the covering space $\tilde{\Gamma}$.  The infinite
  cyclic covering $\tilde{\Gamma}\to\Gamma$ also defines a
  homomorphism $\pi_1\Gamma\to\zee$, which in turn allows elements $g$
  of $\pi_1\Gamma$ to act by translation on $\RR$.
 
  Consider the diagonal action of $\pi_1\Gamma$ on
  $\tilde{\Gamma}\times\RR$.  The quotient is homeomorphic to
  $\Gamma\times\RR$.  Let
  $X=\tilde{\Gamma}\times(-1/2,1/2)\subset\tilde\Gamma\times\RR$. Distinct
  translates of $X$ are disjoint, and so the map
  $X\into\tilde\Gamma\times\RR$ descends to an embedding
  $X\into\Gamma\times\RR$.  Any choice of homeomorphism
  $(-1/2,1/2)\cong \RR$ identifies $X$ with
  $\tilde{\Gamma}\times\RR$. It is straightforward to check that the
  claimed diagram commutes.
\end{proof}

We are now ready to prove that stackings exist.  A very simple example
of a stacking is illustrated in Figure~\ref{fig:example}.

\begin{lemma}
  \label{lem:lifttoembedding}
  Any primitive immersion $\Lambda\colon S^1\to\Gamma$ has a stacking
  \[\hat\Lambda\colon S^1\to\Gamma\times\RR\]
\end{lemma}

\begin{proof}
  Let $\Gamma_0\immerses\Gamma_{1}\immerses\Gamma_n=\Gamma$ be a
  maximal cyclic tower lifting of $\Lambda$, and let $\Lambda_i:S^1\to
  \Gamma_i$ be the lift of $\Lambda$ to $\Gamma_i$. Note that
  $\Gamma_n$ is a circle and $\Lambda_n$ is a finite-to-one covering
  map.  Since $\Lambda$ is primitive, it follows that $\Lambda_n$ is a
  homeomorphism and hence trivially stackable.

  By induction on $n$, we may assume that the map $\Lambda_{n-1}$ has
  a stacking $\hat{\Lambda}_{n-1}:S^1\into\Gamma_{n-1}\times\RR$.  If
  $\Gamma_{n-1}\to\Gamma_n$ is an inclusion of subgraphs then it
  extends naturally to an inclusion
  $i:\Gamma_{n-1}\times\RR\into\Gamma_n\times\RR$, and so
  $\hat{\Lambda}=i\circ\hat{\Lambda}_{n-1}$ is a stacking.
  
  Suppose therefore that $\Gamma_{n-1}\to\Gamma_n$ is an infinite
  cyclic covering map.  Let
  $i:\Gamma_{n-1}\times\RR\to\Gamma_n\times\RR$ be the embedding
  provided by Lemma \ref{lem: Cyclic embedding}.  Then
  $\hat{\Lambda}=i\circ\hat{\Lambda}_{n-1}$ is an embedding
  $S^1\into\Gamma_n\times\RR$, and a simple diagram chase confirms
  that $\hat{\Lambda}_n$ is a lift of $\Lambda_n$.  This completes the
  proof.
\end{proof}

\begin{remark}
  Note that any stacking of a map of a single circle is good for
  trivial reasons.  In fact, Lemma~\ref{lem:lifttoembedding} holds for
  graphs and immersions associated to staggered presentations, and in
  this case as well, $\hat\Lambda$ is necessarily good.
\end{remark}

\begin{figure}[ht]
  \centering
  \includegraphics[width=.7\textwidth]{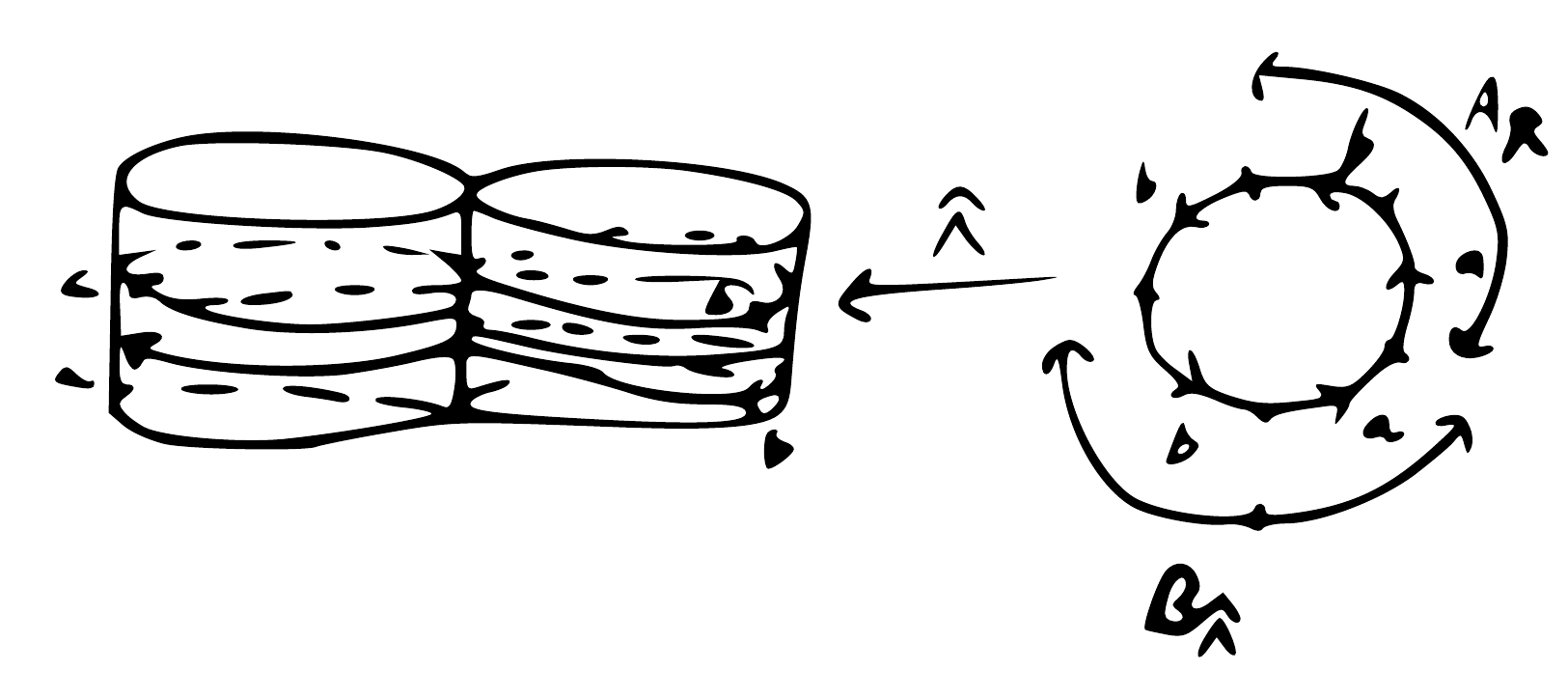}
  \caption{A stacking of the word $aabbb$.}
  \label{fig:example}
\end{figure}

Let $L=\langle x_1,\dotsc,x_n\mid w\rangle$ be a one-relator group,
where $w$ is a cyclically reduced nonperiodic word $w=x_{i_1}\dotsb
x_{i_m}$ in the $x_i$. Duncan and Howie use right-orderability of $L$
to assign heights to the (distinct, by~\cite[Corollary
3.4]{howie:order}) elements $a_0=1$, $a_j=x_{i_1}\dotsb x_{i_j}$,
$j<m$, in $L$ in the same way we use the embedding $\hat\Lambda$ to
find open arcs which remain above ($\AA$) or below ($\BB$) every point
of $S^1$ with the same image in
$\Gamma$. Lemma~\ref{lem:lifttoembedding} is equivalent to the
existence of a right-invariant pre-order on $L$ which distinguishes
between the elements $a_j$.

Our main theorem is now a quick consequence of Lemmas
\ref{lem:reducedrankofsubgroup} and \ref{lem:lifttoembedding}.

\begin{proof}[Proof of Theorem \ref{maintheorem}]
  Let $\Gamma$, $\Gamma'$, etc., be as in Theorem~\ref{maintheorem},
  and let $\hat\Lambda$ be the stacking provided by
  Lemma~\ref{lem:lifttoembedding}. Since $S^1$ is connected, the
  stacking $\hat{\Lambda}$ is automatically good.  By hypothesis
  $\Lambda'$ is not reducible, and therefore by Lemma
  \ref{lem:reducedrankofsubgroup},
  $-\chi(\Gamma')\geq-\chi(\Lambda'(\ess'))\geq\deg\sigma$ as claimed.
\end{proof}

\bibliographystyle{amsalpha} \bibliography{elev}

\providecommand{\bysame}{\leavevmode\hbox to3em{\hrulefill}\thinspace}
\providecommand{\MR}{\relax\ifhmode\unskip\space\fi MR }
\providecommand{\MRhref}[2]{%
  \href{http://www.ams.org/mathscinet-getitem?mr=#1}{#2}
}
\providecommand{\href}[2]{#2}
\begin{thebibliography}{How82}

\bibitem[Bro80]{brod:order}
S.~D. Brodski{\u\i}, \emph{Equations over groups and groups with one defining
  relation}, Uspekhi Mat. Nauk \textbf{35} (1980), no.~4(214), 183. \MR{586195
  (82a:20041)}

\bibitem[DH91]{duncanhowie}
Andrew~J. Duncan and James Howie, \emph{The genus problem for one-relator
  products of locally indicable groups}, Math. Z. \textbf{208} (1991),
  225--237.

\bibitem[How81]{howie_pairs_1981}
James Howie, \emph{On pairs of {$2$}-complexes and systems of equations over
  groups}, J. Reine Angew. Math. \textbf{324} (1981), 165--174. \MR{614523
  (82g:20060)}

\bibitem[How82]{howie:order}
\bysame, \emph{On locally indicable groups}, Math. Z. \textbf{180} (1982),
  no.~4, 445--461. \MR{667000 (84b:20036)}

\bibitem[HW14]{helfer_counting_2014}
Joseph Helfer and Daniel~T. Wise, \emph{Counting cycles in labeled graphs: the
  nonpositive immersions property for one-relator groups}, Preprint, 2014.

\bibitem[Lyn50]{lyndon_cohomology_1950}
Roger~C. Lyndon, \emph{Cohomology theory of groups with a single defining
  relation}, Annals of Mathematics. Second Series \textbf{52} (1950), 650--665.

\bibitem[Wis05]{wise}
Daniel~T. Wise, \emph{The coherence of one-relator groups with torsion and the
  {H}anna {N}eumann conjecture}, Bull. London Math. Soc. \textbf{37} (2005),
  no.~5, 697--705. \MR{2164831 (2006f:20037)}

\end{thebibliography}

\begin{flushleft}
\emph{email:} \texttt{l.louder@ucl.ac.uk, h.wilton@maths.cam.ac.uk}
\end{flushleft}

\end{document}